\documentclass[12pt, reqno]{amsart}
\usepackage{amsmath,amsthm,amsfonts,amsopn,cite,mathrsfs}
\usepackage{fancybox,subfigure,latexsym}

\usepackage[T1]{fontenc}
\usepackage[utf8]{inputenc}
\usepackage{hyperref}
\usepackage{color}
\usepackage{xcolor}

\theoremstyle{plain}
\newtheorem{theorem}{Theorem}
\newtheorem{proposition}{Proposition}
\newtheorem{lemma}{Lemma}
\newtheorem{corollary}{Corollary}
\theoremstyle{definition}
\newtheorem{definition}{Definition}
\theoremstyle{remark}
\newtheorem{remark}{Remark}

\def\opnorm#1#2{{|\nnth \| {#1} | \nnth \|_{#2} \, }}
\newcommand{\Asp}{{\boldsymbol A}}     
\newcommand{\cG}{\mathscr{G}}     
\newcommand{\AspN}{(\Asp, \, \|\ebbes\|_\Asp)}     
\newcommand{\ebbes}{\mbox{$\,\cdot\,$}}     
\newcommand\Astt{{\Asp_{s_2}}}     
\newcommand\Biw{{\Bsp_{1,w}}}     
\newcommand{\Bsp}{{\boldsymbol B}}     
\newcommand\BiwN{{\nspb \Biw }}     
\newcommand\BiwRd{{ \Biw (\Rdst)}}     
\newcommand{\Rdst}{{{\Rst^d}}}     
\newcommand\BPsp{{ \Bsp'}}
\newcommand{\BspN}{(\Bsp, \, \|\ebbes\|_\Bsp)}     
\newcommand{\Bspq}{{\Bsp^s_{\negthinspace p,q}}}     
\newcommand{\BspqRd}{{\Bspqsp(\Rst^d)}}     
\newcommand{\Bspqsp}{{\Bsp^{s}_{\negthinspace p,q}}}     
\newcommand{\Rst}{{\mathbb R}}     
\newcommand\BspqRdN{{ (\BspqRd, \| \ebbes \|_{\Bspq}) }}     
\newcommand{\COsp}{{\Csp_{\negthinspace 0}}}     
\newcommand{\CORdN}{{\big( \COsp(\Rst^d), \, \|\ebbes\|_\infty \big)}}     
\newcommand{\Csp}{{\boldsymbol C}}     
\newcommand{\Ccsp}{{\Csp_{\negthinspace c}}}     
\newcommand{\CcRd}{{\Ccsp(\Rst^d)}}     
\newcommand{\DPsi}{{\operatorname{D}_\Psi}}     
\newcommand{\DRd}{{\Dcsp(\Rst^d)}}     
\newcommand{\Dcsp}{{\boldsymbol{\mathcal D}}}     
\newcommand{\FT}{{\operatorname{{\mathcal F} \negthinspace}}}     
\newcommand{\Lisp}{{\Lsp^1}}     
\newcommand\FLivst{{ \FT \Lsp_{\tiny{v_{s_{ 2}}}}^1}}     
\newcommand{\Lsp}{{\boldsymbol L}}     
\newcommand\FLiw{{\FT \Liwsp}}     
\newcommand{\Liwsp}{{\Lsp^1_{\negthinspace w}}}     
\newcommand\FLiwRd{{\FT \Liwsp(\Rdst)}}     
\newcommand\FLiwRdN{{ \nspb \FLiwRd}}     
\newcommand\FLiwsp{{\FT \Liwsp}}     
\newcommand\nth{\negthinspace}     
\newcommand{\Fsp}{{\boldsymbol F}}     
\newcommand{\Fspq}{{\Fsp^s_{\negthinspace p,q}}}     
\newcommand{\FspqRd}{{\Fspqsp(\Rst^d)}}     
\newcommand{\Fspqsp}{{\Fsp^{s}_{\negthinspace p,q}}}     
\newcommand\FspqRdN{{ (\FspqRd, \| \ebbes \|_{\Fspq}) }}     
\newcommand\Hilb{\mathcal H}     
\newcommand{\LiRd}{{\Lisp \nth (\Rst^d)}}     
\newcommand{\LiRdN}{\big( \LiRd, \, \|\ebbes\|_1 \big)}     
\newcommand\Livsi{{\Lsp_{v_{s_{\tiny 1}}}^1}}     
\newcommand{\LiwRd}{{\Liwsp(\Rst^d)}}     
\newcommand{\LiwRdN}{\big( \LiwRd, \, \|\ebbes\|_{1,w} \big)}     
\newcommand{\Lpsp}{{\Lsp^p}}     
\newcommand{\Msp}{{\boldsymbol M}}     
\newcommand{\MspqRd}{{\Mspqsp(\Rst^d)}}     
\newcommand\Mspqsp{{\Msp^s_{\negthinspace p,q}}}     
\newcommand\MspqRdN{\big( \MspqRd, \, \|\ebbes\|_{\Mspqsp} \big)}     
\newcommand{\Nst}{{\mathbb N}}     
\newcommand\Psifam{{ \Psi = (\psi_i)_{i \in I} }}     
\newcommand{\ScPRd}{{\ScPsp(\Rst^d)}}     
\newcommand{\ScPsp}{{\Scsp'}}     
\newcommand{\Scsp}{{\boldsymbol{\mathcal S}}}     
\newcommand{\ScRd}{{\Scsp(\Rst^d)}}     
\newcommand{\Strho}{{\operatorname{St}_{\negthinspace \rho}}}     
\newcommand\WAstlisi{{\Wsp(\Astt,\ell^1_{v_{s_{\tiny 1}}})}}
\newcommand{\Wsp}{{\boldsymbol W}}     

\newcommand{\Zdst}{{\Zst^d}}     
\newcommand{\Zst}{{\mathbb Z}}     
\newcommand\chck{^\checkmark \negthinspace}
\newcommand{\diam}{\operatorname{diam}}     
\newcommand\epso{{ \varepsilon > 0 }}     
\newcommand{\ghat}{{\widehat{g}}}     
\newcommand\hatf{{\widehat{f}}}     
\newcommand{\hatg}{\widehat{g}}     
\newcommand\hkr{\hookrightarrow}     
\newcommand{\ie}{i.e.}     
\newcommand{\intRd}{\int_{\Rst^d}}     
\newcommand{\inv}{^{-1}}     
\newcommand\japx{{\langle x \rangle}}     
\newcommand\japy{{\langle y \rangle}}     
\newcommand{\kiZd}{{{k \in \Zdst}}}
\newcommand\limal{{\lim_{\alpha \to \infty} \, }}     
\newcommand{\lsp}{{\boldsymbol\ell}}     
\newcommand\livsi{{\lsp^1_{\vsi}}}     
\newcommand\vsi{v_{s_{1}}}     
\newcommand{\nnth}{{ \negthinspace \: \negthinspace }}     
\newcommand\sPsi{|\Psi|}     
\newcommand{\spec}{\operatorname{spec}}     
\newcommand{\sumiI}{\sum_{i\in I}}     
\newcommand{\sumkZd}{\sum_{\kiZd}}     
\newcommand{\supp}{\operatorname{supp}}     
\newcommand\suth{{ \, | \, } }     
\newcommand\tblue{\textcolor{blue}}     
\newcommand\tred{\textcolor{red}}     
\newcommand\veps{{\varepsilon}}     

\def\FLivsp{{\FT \nth \Lsp^1_{v}}}
\def\FLivRd{{\FLivsp(\Rdst)}}
\def\LivRd{{\Lsp^1_v(\Rdst)}}
\def\Bnorm#1{{ \| #1 \|_\Bsp }}
\def\Anorm#1{{ \| #1 \|_\Asp }}

\def\cG{G}

\def\Livsi{{\Lsp_{v_{s_{1}}}^1}}
\def\FLivst{{ \FT \Lsp_{{v_{s_{ 2}}}}^1}}
\def\WAstlisi{{\Wsp(\Astt,\livsi \nnth)}}
\def\LiwRdN{{(\LiwRd, \| \ebbes\|_\Liwsp)}}
\def\FLiwRdN{{(\FLiwRd, \| \ebbes \|_\FLiwsp)}}

\def\nspb#1{{ (#1,\| \ebbes \|_{#1} )}}
\def\ebbes{\mbox{$\,\cdot\,$}}
\def\normta#1#2{{  \| {#1}   \|_{#2} \, }}

\begin{document}

	\title[]
{Completeness of Sets of Shifts in \\
Invariant Banach Spaces of Tempered Distributions \\
via Tauberian conditions}

\author{Hans G. Feichtinger and Anupam Gumber}


\thanks{ \tblue{
Address:   Faculty of Mathematics, University of Vienna, Oskar-Morgenstern-Platz 1,1090 Vienna,\\
and (hgfei) Acoustic Research Institute ARI, OEAW, Wohllebengasse 12, 1040 Vienna, AUSTRIA; \\
ORCID:  0000-0002-9927-0742, \, E-mail:  \, hans.feichtinger@univie.ac.at;  \\
\, \, 
ORCID:   0000-0001-5146-3134, \, E-mail: \,  anupam.gumber@univie.ac.at}}



\vspace{3mm}
\def\tblue{}
\def\tred{}

\begin{abstract}
{The main result of this paper is a far reaching generalization of
the completeness result given by V.~Katsnelson in a recent
paper (\nth \cite{ka19-1}).
 \tblue{ Instead of just using a collection of dilated Gaussians it is shown
that the key steps of an earlier paper (\nth \cite{fegu20}) by the authors, combined with the use of
 Tauberian conditions (i.e.\ the non-vanishing of the
Fourier transform) allow us to show that the linear span of the translates
of a single function $g \in \ScRd$  is a dense subspace of \tred{any} Banach space
satisfying certain double invariance properties.
\newline
{In fact, a much stronger statement is presented: for a given
compact subset $M$ in such a Banach space $\BspN$ one can construct a finite rank
operator, whose range is contained in the linear span of finitely many
translates of $g$, and which approximates the identity operator
over $M$ up to a given level of precision.} \newline
 The setting of
tempered distributions allows to reduce the technical arguments
to methods which are widely used in Fourier Analysis. The extension
to \tred{non-quasi-analytic weights} respectively locally compact Abelian
groups is left to a forthcoming paper, which will be technically much
 more involved and uses different ingredients. }
}
\end{abstract}
\subjclass[2010]{Primary 43A15, 43A10, 41A65, 46F05, 46B50; Secondary 43A25, 46H25, 46A40, 40E05}
\keywords{Beurling algebra, modulation spaces, weighted amalgam spaces, Tauberian theorems, approximation by translations, Banach spaces of tempered distributions, Banach modules}

\maketitle


\section{Introduction}

This paper can be seen as an alternative approach to the question of
completeness of sets of translates of a given test function for a large
variety of Banach space of tempered distributions. The motivation for
the current paper is the wish to demonstrate that the very specific
results describing the density of the linear span of the set of
all translates and dilates of the Gauss function as given in \cite{ka19-1}
can be generalized into several directions.

In the companion paper \cite{fegu20}   we have
shown that only the non-vanishing integral, i.e. without loss of generality
the assumption $\intRd g(x)dx = 1$ for the generating Schwartz function $g \in \ScRd$ is sufficient in order to
guarantee the density for a large variety of
Banach spaces with double module structure. As it turned out, the setting
of the paper \cite{dipivi15-1} appeared to be most appropriate, which is
quite similar to the setting of so-called ``standard spaces'' as used
in papers on compactness (\cite{fe84}) or double module structures (\cite{brfe83}). As it is clear that there is only a chance for such
a statement if $\ScRd$ is a dense subspace of $\BspN$, so we will make this
{\it minimality} assumption throughout this paper.

Observing that the Gauss function has another important property, namely
a nowhere vanishing Fourier transform, Tauberian Theorems come to
mind (see \cite{wi32}, \cite{wi33}, \cite{rest00}, \cite{ko04})
\cite{fe88}). In the classical setting the non-vanishing of the
Fourier transform $\widehat{g}(s) \neq 0$ for all $s \in \Rdst$
for some $g \in \LiRd$ implies that the linear span of its translates
is dense in $\LiRdN$ (see \cite{rest00}, Chap.4.1). Similar results
hold true for weighted $\Lisp$-algebras, the so-called {\it Beurling
algebras} $\LiwRdN$ (see e.g.\ \cite{rest00}, Chap.1.6). For this
paper mostly polynomial weights equivalent to $v_s(x) = (1+|x|^2)^{s/2}$
are of interest. For them the Tauberian Theorem can be derived
from the Tauberian condition going back to the work of A.~Beurling
(\cite{be38}) for  details(see \cite{rest00}, Theorem 1.6.17).

The main result of this paper is a {\it constructive} realization
of an approximation procedure which applies to all members of a family
of  translation and modulation invariant Banach spaces. Since the
intersection of all these spaces under consideration is just the
Schwartz space $\ScRd$ it is natural to start from a Schwartz function
$g_0$ with non-vanishing Fourier transform. We call this key assumption
a {\it Tauberian condition}, because it is well-known from the Wiener's
classical Tauberian Theorem (see \cite{rest00}).

The main advantage of the approach given here is the fact that it does not
make use of dilations, so we do not have to generate Dirac sequences by
compressing the given building block $g \in \ScRd$
(which \tblue{only has to satisfy the condition} $\hatg(0)\neq 0$). Thus at least the formulation
of the main result of the current paper can easily be transferred
to the setting of LCA (locally compact Abelian) groups without
significant changes.

One may also argue that the restrictions to polynomially moderate
weight functions and hence to stay within the world of $\ScPRd$,
the tempered distributions, is somewhat restrictive, and that
one should look for such results also in the context of
ultra-distributions. In \tred{that context even}  more sensitive norms
can be used for the space of test functions (so in a way the
kind of approximation that has to be achieved is much more
challenging), but also much bigger spaces \tblue{are allowed
which are not contained in $\ScPRd$ anymore}.

Overall we have chosen to present our results in the context
of tempered distributions over $\cG = \Rdst$, in order to
make the key steps more clear. \tblue{One can formulate the
same claim  using } the same principle ideas \tblue{for
general LCA groups, making use the Schwartz-Bruhat space,
invoking Tauberian Theorems for Beurling algebras as found
in the book of H.~Reiter } \cite{rest00}. \tblue{However, we will go
for the most general setting, namely  Banach spaces of
ultra-distributions over LCA groups in a subsequent paper.
In order to keep the paper readable we have avoided this
outmost level of generality, mentioning once more that this
is still a very far-reaching extension of the results
in } \cite{ka19-1}.
\tblue{ A list of examples of such Banach spaces is given in \cite{fegu20}.}

\section{Translation and modulation invariant spaces} 

Throughout this paper we will work with
the following {\bf standard assumptions}, similar
to the setting chosen in \cite{dipivi15-1} or \cite{fegu20}:
\begin{definition} \label{mintempstanddef}
A Banach space $\BspN$ is called a {\it minimal tempered standard space}
\tred{(abbreviated as {\it \bf{MINTSTA}},} or equivalently a minimal TMIB in the sense of \cite{dipiprvi19}, \cite{fegu20}),
if the following conditions are valid:
\begin{enumerate} \item  
We assume the following chain of continuous embeddings:
\begin{equation}\label{ScSandw}
   \ScRd \hookrightarrow \BspN \hookrightarrow  \ScPRd;
\end{equation}
\item
$\ScRd$ is dense in $\BspN$ ({\it {minimality}});
\item $\BspN$ is {\it translation invariant}, and for
some $s_1 \in \Nst$ and $C_1 > 0 $ one has
\begin{equation}\label{transl1}
  \|T_x f\|_\Bsp \leq C_1 \japx^{s_1} \|f\|_\Bsp
  \quad  \forall  x \in \Rdst, f \in \Bsp.
\end{equation}
\item  $\BspN$ is \it{modulation invariant}, and for
some $s_2 \in \Nst$ and $C_2 > 0 $ one has
\begin{equation}\label{modul1}
  \|M_y f\|_\Bsp \leq C_2 \japy^{s_2} \|f\|_\Bsp
   \quad \forall y \in \Rdst, f \in \Bsp.
\end{equation}
 \end{enumerate}
\end{definition}   
Here we use the Japanese bracket symbol $\japx$ respectively $\japy$ for the function  $v_s(z) = (1+|z|^2)^{s/2}$,$z \in \Rdst$,  which is also known as
{\it Beurling weight} of {\it polynomial type}, because it
satisfies for $s \geq 0$ the {\it submultiplicativity property}
\begin{equation} \label{submultprop1}
  v_s(x+y) \leq v_s(x) v_s(y), \quad x,y \in \Rdst.
\end{equation}
For any such (polynomial) weight $v_s$ the corresponding weighted
$\Lisp$-space $\Lsp^1_{v_s}(\Rdst)$ is a Banach algebra with
respect to convolution, continuously embedded into $\LiRdN$
(since $v_s(x)\geq 1$) and having bounded approximate units
(obtained by $\Lisp$-norm preserving compression, so-called
Dirac sequences).

Whenever we use generic, submultiplicative   {\it Beurling weights}
(not necessarily of polynomial type) we will
make use of the standard notation $w$.
The corresponding weighted $\Lisp$-spaces  $\nspb \LiwRd$
are Banach algebras with respect to {\it convolution} (we use the
symbol $f \ast g$), called {\it Beurling algebras}.
The natural norm to be used is  (cf. \cite{re68})
$$ \|f\|_\LiwRd = \|f\|_{1,w} : = \|f \cdot w\|_\Lisp, \quad f \in \LiwRd.$$


For the formulation of our arguments the terminology of the theory of
Banach modules will be convenient (see \cite{ri69-1,brfe83,fe84}):
\begin{definition} \label{BanModdef}
A Banach space $\BspN$ is called a {\it Banach module over a Banach algebra}
$\AspN$ if there is a (usually natural) embedding of $\AspN$ into the
operator algebra over $\BspN$, written as $(a,b) \mapsto a \bullet b$,
such that this bilinear (and associative) mapping satisfies
\begin{equation}\label{BanModestdef}
  \Bnorm{a \bullet b} \leq \Anorm{a} \Bnorm{b}, \quad a \in \Asp, b \in \Bsp.
\end{equation}

\noindent
If, in addition, $\BspN \hkr \AspN$ and the module operation is just
the internal multiplication of the algebra $\AspN$, we call $\BspN$
a {\it Banach ideal}.

\noindent
A Banach module is called {\it essential}  if $\Asp \bullet \Bsp$ generates
a dense subspace of $\BspN$.
\end{definition}

For our applications the Banach algebras $\AspN$  have bounded approximate units, i.e.\ there are bounded sequences or nets  $(e_\alpha)_{\alpha \in I}$ of elements in $\AspN$ such that
$$ \lim_{\alpha \to \infty} e_\alpha \bullet a = a, \quad a \in \Asp.$$

In such a case the Cohen-Hewitt Factorization Theorem can be applied and gives
 $$ \Bsp = \Asp \bullet \Bsp = \{ a \bullet b \suth a \in \Asp, b \in \Bsp\}.$$

We only consider Beurling algebras as Banach convolution algebras
(in this case we write $f \ast g$ for the abstract multiplication
$f \bullet g$) acting on a MINTSTA
via convolution (we will speak of {\it Banach convolution modules})
or Banach algebras inside of $\CORdN$ acting on $\BspN$ via
pointwise multiplication (we write $f \cdot g$, or simply $f g$).

Essential Banach ideals in the Banach convolution algebra $\LiRdN$
are exactly the {\it Segal algebras} in the sense of H.~Reiter (\cite{re68})
(also called {\it normed ideals} in \cite{ci69}), see \cite{du74}.
Further references to {\it abstract Segal algebras}
 are given in  papers of J.T.~Burnham \cite{bu72}, see also \cite{fe77}.

The density of $\ScRd$ in $\BspN$ and the conditions (\ref{transl1})
and (\ref{modul1}) imply a double module structure for such
Banach spaces:
\begin{proposition} \label{doublemod3}
Any MINTSTA
$\BspN$ has a double module structure:

\noindent (1)
 $\BspN$ is an essential Banach convolution module over the Beurling algebra    $\Livsi$:
       \begin{equation} \label{BLivisestim}
        \Bnorm{g \ast f} \leq \normta g \Livsi \Bnorm{f}, \quad
           g \in \Livsi, f \in \Bsp,
        \end{equation}
 and for any bounded, approximate unit
 $(e_\alpha)_{\alpha \in I}$ in $\LiwRdN$ one has:
\begin{equation}\label{approxLiw1}
\limal \normta {{ e_\alpha \ast f - f}} \Bsp = 0, \quad \forall f \in \Bsp.
\end{equation}

\noindent
(2)    $\BspN$ is an essential Banach module with respect to pointwise multiplication over the Fourier-Beurling algebra $\Astt := \FLivst$,
       with the corresponding norm estimate. Correspondingly bounded
       approximate units for pointwise multiplication act accordingly.

 \noindent (3)
       The Wiener amalgam space
       $ \nspb \WAstlisi  $  with local component $\Astt$ and global
       component $\livsi$ is continuously and densely embedded into $\BspN$.
  \end{proposition}

The statement in this proposition only summarize results which may
be considered as folklore. Integrated action appears in  a similar form in
  \cite{brfe83}, \cite{fe84} or  \cite{dipivi15-1}, for example.
  For the last (minimality) statement see  \cite{fe81}, or more explicitly
  {\cite[Proposition 6]{fe87-1}}.
 Also recall that the {\it Wiener amalgam space}   $ \nspb \WAstlisi  $
 is the subspace of continuous functions on $\Rdst$ which
 belong locally to the Banach algebra  $\Astt$, which is supposed
 to contain $\DRd$. These spaces can be characterized with the
 help of smooth BUPUs  (bounded partitions of unity) in $\DRd$,  \ie
 $\psi \in \DRd$, with $\sumkZd T_k \psi \equiv 1$, as follows:
 \begin{equation}\label{WAstlisinorm}
   \normta f {\WAstlisi} :=
    \sumkZd  \normta {f \cdot T_k \psi} \Astt \vsi(k) < \infty.
 \end{equation}
 Since this fact was partially motivating our approach and because
 we think that its statement may be interesting for the reader we
 have included it here. However, we will not make use of
 the third statement and thus we do not go into technical details here.

The following lemma is inspired by the investigations  in \cite{fe77}:
\begin{lemma} \label{BeurBanId1}
  Let $\BspN \subset \ScPRd $ be a translation invariant Banach space
  with
  $$\Bnorm{T_xf} \leq w(x) \Bnorm{f}, \quad \forall f \in \Bsp, x \in \Rdst,$$
    for some submultiplicative function $w$ on $\Rdst$.

Then $\BiwN$, the vector space $\Biw =  \Liwsp \cap \Bsp$
is a dense, {\it essential Banach ideal} inside the
Banach convolution algebra $ \LiwRdN$ with respect to the  norm
\begin{equation}\label{normBiw}
 \normta {f} {\Biw}
=  \normta f  {\Liwsp} + \normta f \Bsp, \quad f \in \Biw.
\end{equation}
\tblue{Thus it  is a Banach space with respect to the norm (\ref{normBiw}) and
$\Liwsp \ast \Biw \subset \Biw$  with }
\begin{equation}\label{LiwBiwest}
  \normta {g \ast f} \Biw \leq \normta g \Liwsp \normta f \Biw,
  \quad g \in {\Liwsp}, f \in \Biw,
\end{equation}
and that for any bounded, approximate unit
$(e_\alpha)_{\alpha \in I}$ in $\LiwRdN$:
\begin{equation}\label{approxLiw1}
\limal \normta {{ e_\alpha \ast f - f}} \Biw = 0, \quad \forall f \in \Biw,
\end{equation}
and consequently (by the Cohen-Hewitt Factorization Theorem):
\begin{equation}\label{BiwCohen1}
   \tblue{\Liwsp \ast \Biw = \Biw}.
\end{equation}
If $w$ is a Beurling weight of polynomial growth, then
$\ScRd$ is  dense in  $\BiwRd$, hence $\DRd = \Scsp \cap \CcRd$
is a dense subspace of $\BiwN$.
\end{lemma}
\begin{proof} The facts collected in the above lemma are essentially
based on more detailed considerations published in \cite{fe77}.
The first step is to verify that $\Biw$ is a {\it dense} subspace of
$\LiwRd$. For this purpose it is enough to note that compactly
supported elements are dense in any Beurling algebra, i.e.
for given $\epso$ one finds $\varphi \in \CcRd$ with
\begin{equation}\label{CcRdappr2}
  \normta {f - \varphi}  \Liwsp < \veps/4.
\end{equation}

By smoothing $\varphi \in \CcRd$ with a sufficiently small supported
test function $\psi \in \DRd$ one has
\begin{equation}\label{CcRdappr2b}
  \normta {\varphi - \psi\ast \varphi}  \Liwsp < \veps/4.
\end{equation}
Observing that
$ \psi \ast \varphi \in \DRd \subset \ScRd \subset \Bsp$
we note that
$ \psi \ast \varphi \in \Biw$ and satisfies
\begin{equation}\label{CcRdappr2c}
  \normta {f - \psi \ast \varphi}  \Liwsp < \veps/2.
\end{equation}

Since both $\LiwRdN$ and $\BspN$ are translation invariant
and have continuous translation the same is true for $\BiwN$.

The remaining consequences are just stated for easier reference
but are standard results. The last step is a consequence
of the Cohen-Hewitt Factorization Theorem (\cite{hero70}, Chap.32).
\end{proof}

Next let us recall the important Beurling-Domar condition, which
ensures the existence of band-limited elements (i.e.\ functions
with a compactly supported Fourier transform) in general Beurling
algebras. According to \cite{re68} (and referring to the
important paper by Y.~Domar, \cite{do56}) a Beurling weight
$w$ satisfies the Beurling-Domar {\it non-quasianalyticity condition}
if one has
\begin{equation}\label{BDcond00}
  \mbox{(BD)} \quad \sum_{n \geq 1}  \frac{w(nx)}{n^2} < \infty \quad \forall x \in \Rdst.
\end{equation}
It is an easy exercise to check that any {\it polynomial weight}  satisfies
the (BD)-condition.
%
%
One of the advantages of increasing, radial symmetric weights,
\tblue{ such  as $v_s$ for $s \geq 0$} is
the fact that one can create easily bounded approximate units
simply by applying to an arbitrary $g \in \ScRd \subset \Livsi$
with $\hatg(0)=1$ the  $\Lisp$-norm preserving {\it compression operator}
\begin{equation} \label{Strhodef04}
\Strho g(x) = \rho^{-d} g(x/\rho), \quad \rho \to 0.
\end{equation}

The next result is just a reminder concerning Fourier-Beurling algebras,
i.e.\ pointwise Banach algebras obtained as $\FLiwRdN$, with the
norm $ \normta \hatf \FLiw = \normta f \Liwsp$. The statements
of the following lemma are essentially restatements of facts
found in \cite{re68}.
\begin{lemma} \label{Biwbandl1}
\tred{ Let $w$ be a weight function satisfying the (BD)-condition. Then } \newline
 i)   the Beurling algebra   $\LiwRdN$ has bounded approximate units
  consisting of band-limited elements. \newline
 ii) the Banach algebra $\FLiwRdN$ is a {\it Wiener algebra}
  in the sense of H.~Reiter (\cite{re68}), meaning that it
  is a {\it regular}  Banach algebra of continuous functions,
  which allows (among others) to separate compact sets from open
  neighborhoods (see Chap.2.4 of \cite{re68} resp.\ \cite{rest00}).
  In particular, the compactly supported elements are dense. \newline
   \vspace{1mm}
  \noindent
    iii) for weight functions of polynomial growth   one has in addition: \newline
  $ \ScRd \hkr \LiwRdN$ and for any $g \in \ScRd$ with $\intRd g(x)dx =1$   one has
 $$ \| \Strho g \ast f - f\|_\LiwRd \to 0, \,\,\, \mbox{for} \,\,\, \rho \to 0. $$
\end{lemma}
The third claim is essentially Lemma 1 of \cite{fegu20},
see also Corollary 1 of \cite{dipivi15-1}, see also \cite{rest00},
Proposition 1.6.14.
\begin{lemma} \label{bandldensiw}
For any MINTSTA $\BspN$ the band-limited elements form a dense subspace of $\BspN$.
The same is true for $\BiwN$, which is in fact also a MINTSTA itself.
\end{lemma}
\begin{proof}
Since $\ScRd$ is dense in $\LiwRd$ for the (polynomial) weights appearing
in Proposition \ref{doublemod3} the (BD)-condition is clearly satisfied.
Hence there are (even bounded) approximate units in $\LiwRdN$ which
are band-limited (i.e.\ compactly supported on the Fourier transform side),
we write $(e_\alpha)_{\alpha \in I}$.

Since $e_\alpha \ast f$ approximates $f \in \Bsp$
according to (\ref{approxLiw1}) and
$$\supp(\FT(e_\alpha \ast f)) \subset \supp(\widehat{e_\alpha})$$
we find that the band-limited elements are dense in $\BspN$.
\end{proof}

\tblue{For the proof of the main result we will make use of the following
key observations already explained in detail in \cite{fegu20},
using so-called BUPUs, i.e. {\it (bounded) uniform partitions
of unity of size} $|\Psi| = \delta > 0$, i.e. of countable collections
of continuous functions $\psi_i, i \in I$, with $0 \leq \psi_i(x)$,
$\supp(\psi_i) \subset B_\delta(x_i)$ and $\sumiI \psi_i(x) \equiv 1$. }

For simplicity we only consider the regular case, i.e.\ BUPUs which are obtained as translates of a single function:
\begin{definition} \label{regBUPUdef}
A  (countable) {\it family of translates}
$\Psi = (T_\lambda \psi)_{\lambda \in \Lambda}$,
where $\psi$ is a compactly supported function (i.e.\ $\psi
\in \CcRd$), and $\Lambda =  \Asp (\Zdst)$ a lattice in $\Rdst$ (for
some non-singular $d \times d$-matrix $\Asp$) is called a {\bf regular
BUPU} (or more precisely a $\Lambda$-regular BUPU,
or a $\Lambda$-invariant BUPU)  if
\begin{equation} \label{LamregBUPU1}
 {  \sum_{\lambda \in \Lambda}  \psi(x - \lambda) \equiv 1}.
\end{equation}
We will write   $\diam(\Psi) \leq \gamma$ if $\supp(\psi) \subset B_\gamma(0)$
for some $\gamma \tred{\to} 0$.
\end{definition}

We will be interested in the case of ``finer and finer'' BUPUs, i.e.\, the
case $\diam(\Psi) \to 0$. Sometimes we will simply write $|\Psi|$ instead
of $\diam(\Psi)$. The use of fine partitions of unity is also well
established in the context of usual distribution theory, see e.g.
\cite{wa74}.
The following results are required in the sequel, we refer \cite{fegu20} for  details of the proof. The method can even be used to introduce convolution
of measures   (see \cite{fe16}) and goes back to the discretization introduced
in \cite{fe91-1}. 

\begin{proposition} \label{discrconvBeurl1}
There exists $C_1 > 0$ such that for
any BUPU $\Psifam$ with $\sPsi \leq \delta \leq 1$ the mapping
$f \mapsto \DPsi f$, given by
\begin{equation}\label{DPsimudef06} \tblue{
\DPsi f = \sumiI c_i \delta_{x_i} \quad \mbox{with}
\,\, \,\,  c_i = \intRd f(x)\psi_i(x) dx, }
\end{equation}
satisfies
\begin{equation}
 \sumiI  |c_i| w(x_i) \leq C_1 \normta{f} {\Liwsp}, \,
 f \in \LiwRd.  \end{equation}
Moreover, given $g \in \Liwsp$ and $\epso$ there
exists $ \delta \in (0,1]$ such that $\sPsi \leq \delta$ implies
\begin{equation} \label{convappr02}
\normta {g \ast f - g \ast \DPsi f} \Liwsp < \veps.
\end{equation}
\end{proposition}

\vspace{2mm}
\begin{proposition} \label{Liwmod1}
Since any MINTSTA  $\BspN$  is an essential Banach module over some Beurling
algebra $\LiwRdN$ one has for any  $h \in \LiwRd$: Given $\epso$ there
exists $\delta > 0$ such that $|\Psi| \leq \delta$ implies:
\begin{equation}\label{convappr01}
\Bnorm{ g \ast h - g \ast \DPsi h } \leq \veps \normta{h} {\Liwsp}.
\end{equation}
\end{proposition}
\begin{remark} \label{Prop3unif}
It is important to note (for the proof of our main theorem)
that the convergence is uniform for bounded subsets of $\LiwRd$,
because it depends only on the continuous translation property
in $\BspN$, and thus on the quality of $g \in \Bsp$.
\end{remark}

\vspace{3mm}

\section{Preparing the Ground}

Preparing for the main result of this note we start
by recalling the  characterization of compact
subsets in spaces with a double module structure as described in \cite{fe84}.
We will  rephrase the result given there in a form which
is more suitable for the current setting. It just makes use of the fact
that approximate units can always be chosen from a dense subset of the Banach
algebra acting on $\BspN$ (by convolution resp.\ pointwise multiplication):
\begin{theorem} \label{compsetchar}
A closed and bounded subset $M$ of a MINTSTA is compact in $\BspN$ if and only
it is {\it tight} and {\it equicontinuous}, which means, that
for any $\epso$ there exist a band-limited function $g \in \ScRd$
and a compactly supported function $h \in \ScRd$ such that
  $ f \mapsto g \ast f$ and $f \mapsto h \cdot f$
are bounded operators on $\BspN$ and satisfy
\begin{equation}\label{tightchar1}
  \Bnorm{ h \cdot f - f} \leq \veps, \quad \forall f \in M,
\end{equation}
and
\begin{equation}\label{equichar1}
    \Bnorm{ g \ast f - f} \leq \veps, \quad \forall f \in M.
\end{equation}
\end{theorem}

\begin{remark}
The assumptions boil down to choose appropriate elements from
a Dirac sequence $(g_n)_{n \geq 1}$, e.g.\ to choose $g$ as a  suitable
compression of a band-limited function $g_0$ with
$\widehat{g_0}(0)=1$ and $\supp(\widehat{g_0})$ compact.
Correspondingly one can choose for pointwise
multiplication from a  sequence of  (inverse)
Fourier transforms of another Dirac sequence, i.e.\ $ h_n := \FT\inv(g_n)$,
but now with compact support on the time side.

Note that approximate units can be chosen from any dense subspace
of $\LiwRdN$. The most useful examples are band-limited Dirac sequences
and compactly supported pointwise approximate units (which help to
localize functions).
\end{remark}

For the proof of our main result we will also need the following
useful lemma:
\begin{lemma} \label{tightpres1}
Given a double Banach module  $\BspN$, over a Beurling algebra $\LiwRd$
with respect to convolution and over $\FLivRd$ under pointwise
multiplication. Assume that $M$ is a bounded and tight subset of $\BspN$
and  $S$ a bounded, tight subset  of $\LiwRdN$. Then the set
$ S \ast M$ is a bounded and tight subset of $\BspN$.
\end{lemma}
\begin{proof}
Without loss of generality 
$\sup_{g \in S} \normta{g} {\LiwRd}  \leq 1$
and $\sup_{f \in M} \normta{f} {\Bsp}  \leq 1$.
The estimate
$\Bnorm{g \ast f} \leq \normta {g} {\Liwsp} \cdot
   \Bnorm{f} \leq 1$
shows that $S \ast M$ is bounded  in $\BspN$.

Since $\BspN$ is a Banach module over the Fourier-Beurling
algebra $\FLivRd$ (for some weight of polynomial growth) we
may assume that there is a compactly supported (pointwise)
approximation to the identity $(k_\alpha)_{\alpha \in I}$,
with $\normta {k_\alpha} {\FLivsp} \leq C_v < \infty$ (usually
$C_v = 1$).
Since $\LivRd$ contains $\ScRd$ it is clear that $\FLivRd$
also contains compactly supported functions $h \in \ScRd$ which are
plateau-like, i.e.\ which satisfy $h(x) \equiv 1$ on a given compact
set $Q$. Thanks to \cite{baparo72} we may combine these two properties
and observe that there is a bounded approximate unit
$(k_\alpha)$  in $\FLivRd$ with $\normta {k_\alpha} {\FLivRd} \leq C_v$
for all $\alpha$, and with the extra property
that for any compact set $Q \subset \Rdst$ one finds some
index $\alpha_0$ such that $ k_\alpha(x) \equiv 1 $ on
$Q$ for all $\alpha \succeq \alpha_0$.

Given  now  $\epso$, the tightness assumptions allow to
find  $k_1,k_2 \in \CcRd$  with
$$
\normta{ k_1 \cdot g - g} \Liwsp \leq \veps/(2\cdot C_v), \quad g \in S,
$$
and
$$
\Bnorm{k_2 \cdot f - f}  \leq \veps/2, \quad f \in M.
$$
Combining these estimates we find that
\begin{equation}\label{tightest2}
  \Bnorm{ (k_1\cdot g) \ast (k_2 \cdot f) - g \ast f} \leq
\Bnorm{(k_1 \cdot g - g) \ast (k_2 \cdot f)} +
\Bnorm{ g \ast (k_2 \cdot f - f)}.
\end{equation}
The first term can be estimated as
\begin{equation}\label{tightest3}
\leq \normta{ k_1 \cdot g - g} \Liwsp \Bnorm{k_2 \cdot f}
\leq (\veps/(2\cdot C_v)) \normta{k_2} {\FLivsp} \Bnorm{f} \leq \veps/2,
\end{equation}
and the second term  in a similar way.
Combining these estimates we find that
\begin{equation}\label{tightest4}
\Bnorm{(k_1\cdot g)\ast (k_2\cdot f)-g\ast f}\leq \veps,
\quad g \in S, f \in M.
\end{equation}
Now it is obvious that the elements in $\Bsp$ of the
form $$ \{ k_1\cdot g) \ast (k_2 \cdot f), g \in S, f \in M \}$$
have common compact support inside  the compact set
$ Q = \supp(k_1) + \supp(k_2)$.
As explained above
we can find now some $k_3 \in \FLivRd$, with $$\normta {k_3} \FLivRd
= \normta {\widehat{k_3}} {\Lsp^1_v}
\leq C_v,
\quad \mbox{and} \quad k_3(x)
 \equiv 1 \,\,\mbox{on} \,\, Q. $$
Since $\BspN$ is a pointwise Banach module over
$\FLivRd$  we also have
\begin{equation}\label{FLivmod2}
  \normta{ k \cdot h} \Bsp \leq \normta k {\FLivRd} \Bnorm{h},
  \quad  k \in \FLivRd, \,  h\in \Bsp.
\end{equation}
Multiplying the terms in estimate (\ref{tightest4})
by $k_3$ (the first one is unchanged!), we obtain
\begin{equation}\label{tightest5a}
\Bnorm{(k_1\cdot g)\ast (k_2\cdot f)- k_3\cdot(g\ast f)}
\leq \normta {\widehat{k_3}} {\Lsp^1_v} \cdot
\Bnorm{(k_1\cdot g)\ast (k_2\cdot f)-g\ast f}
\leq C_v \, \veps.
\end{equation}
Hence we get the final estimate by combining estimates (\ref{tightest4})
and (\ref{tightest5a}):
\begin{equation}\label{tightest5b}
\Bnorm{ k_3\cdot(g\ast f) - g \ast f}\leq (1+ C_v) \, \veps,
\quad  g \in S, f \in M.
\end{equation}
 \end{proof}

Another estimate which we will need, but which might be useful
as a separate result elsewhere, is formulated in the following lemma.

In the course of the proof of our main result we will
also make use of the following observation:
\begin{lemma} \label{Liwbdn04}
If $\BspN \hkr \ScPRd$ is a translation
invariant Banach space with
$$ \Bnorm{T_xf} \leq w(x) \Bnorm{f}, \quad \forall f \in \Bsp. $$
Then for any pair of functions $g_3 \in \ScRd$
and $k \in \CcRd$ the product-convolution operator
$f \mapsto  k \cdot(g_3 \ast f)$
defines a bounded operator from $\BspN$ into $\LiwRd$.
\end{lemma}
\begin{proof}
By assumption any $g_3 \in \Bsp \subset \ScPRd $
defines a bounded linear functional  on $\BspN$,
which (at least for $f \in \ScRd \subset \Bsp$)
allows to interpret the convolution for any $f \in \Bsp$ in a pointwise sense:
\begin{equation}\label{convBScP}
   g_3 \ast  f(x) 
    = \intRd f(z + x)  g_3(-z) dz
     = \sigma_{3} (T_{-x}f),
\end{equation}
where $\sigma_3$ means the tempered distribution
induced by $g_3 \chck \in \Bsp$
(using $g\chck(x):= g(-x)$).
By the continuity of both sides we can thus use the identity
\begin{equation}  \label{repr3}
 g_3 \ast f(x) = \sigma_3 (T_{-x}f),
\quad f \in \Bsp, \, x \in \Rdst.
\end{equation}
Since translation is continuous, this implies that $g_3 \ast f$
is a continuous function (of polynomial growth). As a pointwise
estimate we have
\begin{equation}\label{convScPpt1}
  |g_3 \ast  f(x)| \leq \normta{g_3} {\BPsp}   \Bnorm{T_{-x}f}
   \leq  \normta{g_3} {\BPsp}  \opnorm{T_{-x}} \Bsp \Bnorm{f}.
\end{equation}
Since $K :=\supp(k)$ is a compact set, the weight function
is bounded on $-K$, or
$$  \sup_{x \in K}  w(-x)  = C_{k} < \infty $$
which implies for any $x \in \supp(k)$ the pointwise estimate
\begin{equation}\label{convScPpt2}
  |k(x)| |g_3 \ast  f(x)|
   \leq  w(-x)  \normta{g_3} {\BPsp}    \Bnorm{f}
   \leq  C_k \normta{g_3} {\BPsp}    \Bnorm{f}.
\end{equation}
This in turn provides us with the required estimate:
\begin{equation}\label{convScPpt3}
\normta {|k| \cdot |g_3 \ast  f|} {\Liwsp} =
 \intRd |k(x)| |g_3 \ast  f(x)| w(x) dx \leq
  [ C_k \normta {k} {\Liwsp} \normta{g_3} {\Bsp'}]   \Bnorm{f}.
\end{equation}
\end{proof}

For the proof of the main result we will make use of the following
	key observations already explained in detail in \cite{fegu20},
	using so-called BUPUs, i.e. {\it (bounded) uniform partitions
		of unity of size} $|\Psi| = \delta > 0$, i.e. of countable collections of continuous functions $\psi_i, i \in I$, with $0 \leq \psi_i(x)$,
	$\supp(\psi_i) \subset B_\delta(x_i)$ and $\sumiI \psi_i(x) \equiv 1$.

The following technical lemma will be useful for our considerations.
It requires the notion of {\it well-spread} sets, playing a 
big role for the use of BUPUs for the characterization 
of {\it Wiener amalgam space} (see \cite{fe83}) or in {\it coorbit theory}
(see \cite{fegr89}).
\begin{definition}
	A discrete family $X= (x_i)_{i \in I}$ is {\it relatively separated}
     if for some (and thus for any) $r > 0$ the number of points from $X$ is uniformly bounded, 	i.e.\ if
	$$ \sup_{x \in \Rdst}  \# \{ i \in I \suth x_i \in B_r(x) \} \leq C_r < \infty.$$
\end{definition}
\begin{lemma} \label{exBUPU01}
	For any family $(x_j)_{j \in J}$ which is $\delta/3$ dense, there exists
	a subfamily $(x_i)_{i \in I}$ which is well-spread and $\delta$-dense.
	Moreover, there exists corresponding BUPUs (in the sense of \cite{fe83}) 
$(\psi_i)_{i \in I}$ 	of size $\delta$, with $ \supp(\psi_i) \subseteq B_\delta(x_i), i \in I$. 
\end{lemma}
\begin{proof}
The generic approach to this task is to find a maximal subfamily $I \subset J$
	such that one has a pavement of size $\delta/3$ in $\Rdst$,
       i.e.\ such that the balls
	$B_{\delta/3}(x_i)$ are pairwise disjoint for $i \in I \subset J$.
	The maximality implies that for any $ j \in J \setminus I$ the
	corresponding ball $B_{\delta/3}(x_j)$ has nontrivial intersection
	with one of the balls centered at $x_i$, for some $i \in I$. As the
	original family (running through the index set $J$) covers all of $\Rdst$
	it is an immediate consequence of the triangle inequality that the family
	$B_{2\delta/3}(x_i), i \in I$ covers all of $\Rdst$, and the same is true
	for $(B_\delta(x_i))_{i \in I}$.

In order to estimate the number of points $x_i, i \in I$ which are inside of $B_r(x)$ we just have to observe that
$x_i \in B_r(x)$ implies that $B_{\delta/3}(x_i) \subset B_{r+\delta}(x)$. The
disjointness of the sets constituting the pavement and the uniform
volume estimate for balls of radius $r+\delta$ (independent from $x$) then
provides the required counting estimate, uniformly with respect to $x$. 
	
	In order to create a BUPU of size $\delta$ with $\supp(\psi_i) \subset B_\delta(x_i)$ we proceed as follows:
	First we observe that the selected family of balls has finite overlap.
	Moreover, starting from some function $\varphi \in \CcRd$ with
	$\varphi(z) = 1$ on $B_{2\delta/3}(0)$ and $\supp(\varphi) \subseteq B_\delta(0)$, we set $\Phi = \sum_{i\in I} T_{x_i} \varphi $ and put
	$\psi_i \equiv T_{x_i} \varphi / \Phi$. This implies $\sumiI \psi_i(x) \equiv 1$
	and
	$$ \supp(\psi_i) \subseteq \supp(T_{x_i}\varphi)
	= x_i +\supp(\varphi) \subset B_{\delta}(x_i).$$
\end{proof}
\begin{remark}
	For the Euclidian set there is an alternative method to select a well-separated
	set from a given $\delta/3$-dense family $(x_j)_{j\in J}$. Starting from
	a partition of $\Rdst$ with cubes of size $\alpha > 0$ (small enough),
	centered at the lattice $\alpha \Zdst$, one can pick one point  within
	each such cube. We leave the details to the interested reader.
\end{remark}

\section{The Main Result } 

\begin{theorem} \label{TaubKats20E}
Given a {relatively compact set} $M$ in any MINTSTA  $\BspN$ and
some  $g_0\in \ScRd$ with $\ghat_0(y) \neq 0$ for all $y \in \Rdst$,
we can show the following:

For any given $\epso$ there exists some $\delta > 0$ such that
for any $\delta$-dense set $(x_i)_{i \in I}$ one can pick a finite
subset $F \subset I$ and construct  a finite rank operator $T$ with range
in the linear span of $ S(g_0):= \{T_{x_i}g_0, i \in F \}$  with
\begin{equation}\label{Mapproxtrans}
  \Bnorm{T(f) - f} \leq \veps, \quad \forall f \in M.
\end{equation}
\end{theorem}

\begin{remark}
We will give  the proof   for compact sets. In fact, since all the estimates allow the transition
from $M$ to the closure $\overline{M}$ of the set $M$ in $\BspN$ (and obviously the
validity for $\overline{M}$ implies the validity for $M$) this implies the general case.
\end{remark}

\begin{proof}
The proof requires a couple of steps. Let us fix $\epso$.
\begin{enumerate}
  \item First of all one has to check that $S(g_0)$ is a subset of $\Bsp$.
  In fact, we have $g_0\in \ScRd \subset \Bsp\cap \LiwRd \subset \Bsp$,
   and all the involved spaces are translation invariant.
   Thus any finite rank operator with range in a finite-dimensional
   subspace of the linear span of $S(g_0)$ will be bounded
   on $\BspN$\footnote{We  cannot expect that there is a uniform bound
   on the operators constructed below!}.
  \item 
   For the given $\epso$ we choose first (according to (\ref{equichar1})
   in Thm. \ref{compsetchar}) some band-limited function
     $g \in \ScRd \subset \Biw$    such that
    \begin{equation}\label{appr01}   
      \Bnorm {f - g \ast f  } < \veps/4, \quad \forall f \in M.
    \end{equation}
     \item
Invoking now the \tred{Tauberian condition and the smoothness of $\hatg_0$}
  we observe that for   such a band-limited   $ g \in \ScRd \subset \BiwRd \subset \LiwRd$
  with $\spec(g) :=  \supp(\hatg) = Q_1 $ (some compact
 set) we can find  another band-limited function
 $ h_1 \in \ScRd$ with compact spectrum $Q_2 = \supp(\widehat{h_1})$,
 such that $ g = g\ast h_1$ (or $\widehat{h_1}(s) = 1$ on $Q_1$).
 Since $\hatg_0(y) \neq 0$ for $y \in Q_2$  it follows
 that $1/\hatg_0(y)$ is a smooth function as well, and
 consequently the pointwise product $ \widehat{h_1}(y)/\hatg_0(y)$
 defines a compactly supported and smooth function in $\ScRd$, whose
 inverse Fourier transform can be called $g_1$.
 Thus we have
  $ \hatg_1(y) \cdot \hatg_0(y) \equiv 1$ on $Q_1$, or
 $$   g =  g \ast g_1 \ast g_0 = g_0 \ast g_1 \ast g \,\,\, \mbox{in} \,
    \, \ScRd. $$
This means that we can factorize $f \ast g$ through $g_0$, for any $f \in \Bsp$, as
\begin{equation} \label{facthg2}
   g \ast f = f \ast g  =  f \ast (g \ast g_1) \ast g_0 =  g_0 \ast (g_1 \ast g \ast f),
  \end{equation}
  in the sense of distributions, with
  $g_1 \ast g \ast f  \in  \LiwRd \ast \Bsp \subset \Bsp,$
  in view of Lemma~\ref{BeurBanId1}.
\item
Next we observe that the tightness of the set $M \subset \Bsp$
implies that also
$ g_1 \ast g \ast M$ is a tight set in $\BspN$
 (by choosing $S = \{ g_1 \ast g\}$
in Lemma \ref{tightpres1}).  Thus we can find
some compactly supported function $k \in \CcRd$ such that
\begin{equation} \label{estim06b}
 \Bnorm{ k \cdot (g_1 \ast g \ast f) - (g_1 \ast g \ast f) }
\leq \veps/{4 \normta {g_0} \Liwsp}, \quad \forall f \in M.
\end{equation}
We also note that by  Lemma \ref{Liwbdn04} (with $g_3 = g \ast g_1$)
\begin{equation}\label{estim06c}
  S_k(M):= \{ h := k \cdot (g_1 \ast g \ast f), f \in M \}
\end{equation} is a bounded subset of $\LiwRd$, so that
Proposition \ref{Liwmod1} can be invoked. Using Lemma \ref{exBUPU01} 
we can apply the discretization operator for any   BUPU (and we can use
the given well-spread family $(x_i)_{i \in I}$ for its centers,
 as long as the set is $\delta-$dense, with $\delta > 0$ {\it only depending
on} $g_0 \in \ScRd \subset \Bsp$) and will get:
\begin{equation}\label{estim07}
  \Bnorm{g_0 \ast h - g_0 \ast \DPsi h} \leq \veps/4,
   \quad \forall h \in S_k(M).  
\end{equation}
Recalling identity (\ref{facthg2}) and the estimate (\ref{appr01}) we note that
\begin{equation}\label{estim04c}
  \Bnorm{f - g_0\ast (g_1 \ast g \ast f)} =  \Bnorm{f - g \ast f}
   \leq \veps/4,   \quad f \in M,
\end{equation}
and furthermore, using the estimate (\ref{estim06b}), one obtains for $f \in M$:
\begin{equation}\label{estim04e}
 \Bnorm{g_0 \ast h - g_0 \ast (g_1 \ast g \ast f)}
 \leq   \normta{g_0} {\Liwsp} \Bnorm{h - g_1 \ast g \ast f} \leq  \veps/4.
\end{equation}

Applying the triangular equation to the last two estimates,
i.e. using
\begin{equation}\label{estim04k}
\Bnorm{f - g_0 \ast h }
\leq \Bnorm{f - g_0 \ast (g_1 \ast g \ast f)} +
   \Bnorm{g_0 \ast (g_1 \ast g \ast f - h) }
\end{equation}
we arrive at the estimate
\begin{equation}\label{estim04f}
\Bnorm{f - g_0 \ast h}  
   \leq \veps/4 + \veps/4 = \veps/2, \quad f \in M.
\end{equation}

So finally combining the estimates (\ref{estim04f}) and (\ref{estim07})
we arrive at our final estimate:
\begin{equation}\label{finalest}
  \Bnorm{f - g_0 \ast \DPsi h} \leq  \Bnorm{f - g_0 \ast h} +
   \Bnorm{g_0 \ast h - g_0 \ast \DPsi h} \leq \veps, \quad f \in M.
\end{equation}

Writing the approximation operator  $T$ in explicit form, we have
\begin{equation}\label{approxop007}
  Tf =   {g_0 \ast \DPsi h} =
  g_0 \ast \DPsi(k \cdot(g_1 \ast g \ast f)) =
    \sum_{i \in F}  c_i T_{x_i} g_0,
\end{equation}
  with the coefficients $(c_i)_{i \in F}$ depending in
  a linear way on the function $f \in \Bsp$ via
\begin{equation} \label{approxop008}
   c_i :=   \intRd h(x) \psi_i(x) dx =
   \intRd (k \cdot(g_1 \ast g \ast f))(x) \psi_i(x) dx , \quad i \in F.
\end{equation}
Here, due to the fact that the centers are well-spread and that
$\supp(k)$ is compact  the following set $F$ is a finite: 
\begin{equation}\label{finiteset07}
   F  := \{i \in I \suth  \psi_i \cdot k \neq 0\}.
\end{equation} 
\end{enumerate}
 \end{proof}

\begin{remark}  \label{mostgeneral}
A careful analysis of the proof given above shows that the assumption
$g \in \ScRd$ is convenient in the given setting, but in fact it
is only required to assume $g \in  \BiwRd = \Bsp \cap \LiwRd $.
For such a case the pointwise inversion argument based on
smoothness used in step (3) of the above has to replaced by
Wiener's inversion Theorem for Beurling algebras  which
has been already mentioned earlier (see \cite{rest00}, Chap.1.6.5).
Since in this case the possible choices of $g$ depend on $\Bsp$
(and thus is not universal with respect to the family of
Banach spaces under consideration) we have chosen to present
our results in the form of Theorem \ref{TaubKats20E}.
\end{remark}

\begin{remark} \label{maincmt}
The proof of the main theorem demonstrates a couple of facts: While ideally
one would like to replace the given functions $f \in M$ by functions which are
both band-limited (hence smooth) and compactly supported, the impossibility
of getting this done at the same time requires to carry out the corresponding
modifications of the given functions $f \in M$ stepwise, in order to come
up with the final approximation, obtained via a discrete convolution.

In order to make things work one has to make use of the fact that convolution
typically improves {\it local} properties (such as continuity or smoothness),
while preserving global decay conditions. In a similar way
pointwise multiplication by smooth, compactly supported
functions will preserve the local properties, while improving
the {\it global} ones (e.g. by creating a compactly supported function).

The structure of the proof would not really simplify if we had 
restricted our attention to the unweighted case. Even for simple
special cases such as $\Bsp = \Lpsp(\Rdst)$ the key steps would be the same.
A crucial property used is the uniformity of the corresponding 
approximation arguments over (relatively) compact sets $M$ in $\BspN$.
\end{remark}

\section{Applications}

Although we have indicated that there is room for further generalization
we have to point out that the list of examples to which the above result
applies is   endless. The first author has tried to collect the construction
principles widely (or occasionally) used in Fourier Analysis in a systematic
way. It appears to be harder to identify space which are generally important,
containing $\ScRd$ as a dense subspace, but not satisfying the assumptions
formulated for this paper.

In order to mention at least some of the most important cases which are
playing a major role in the literature let us mention:
\begin{enumerate}
  \item Weighted $\Lpsp$-spaces, for $1 \leq p < \infty$ with
    polynomial weights (see \cite{fe79},\cite{gr07});
   \item Wiener amalgam spaces of the form $\Wsp(\Lpsp,\ell^1_{v_s})$,
  for $1 \leq p,q < \infty$ (see \cite{he03}), but also
  weighted Wiener amalgam spaces, as treated  in \cite{grheok02};
  \item The 
  Besov-Triebel-Lizorkin spaces $\BspqRdN$ resp.\
  $\FspqRdN$, for $1 \leq p,q < \infty$ (see \cite{pe76} or the
  books of H.~Triebel \cite{tr83});
  \item Modulation spaces in general (see \cite{fe03-1}), but
  specifically the classical modulation spaces  $\MspqRdN$, also
  for the case   $1 \leq p,q < \infty$. See \cite{beok20} and
  \cite{coro20};
  \item More general decomposition spaces, as discussed in the
  literature, based on the approach given in \cite{fegr85};
  \item Tauberian Theorems have been derived directly in the context
  of functions of bounded means in \cite{fe88}, where they have
  been a crucial step in order to extend Wiener's Third Tauberian
  Theorem to the full range $1 <  p \leq \infty$,
  and for $\Rdst$ instead of just $p=2$ and $d=1$  (see \cite{wi32,wi33});
  \item The atomic space $\Hilb(q,p,\alpha)$ appearing in   \cite{fefe19}
  as well as many other atomic spaces, including the exotic case discussed
  in \cite{fezi02}, are  {MINTSTAs};  
  \item General construction principles for a further variety of function
  spaces as well as a long list of references are provided in \cite{fe15}.
\end{enumerate}

The fact that the Schwartz space can be characterized as the intersection
of modulation spaces (see \cite{gr01}, Proposition 11.3.1) implies that
the arguments used for the proof of our main result also provide a constructive
way to verify that the set of translates of $g$ generates a dense
subspace of the Schwartz space $\ScRd$.
\begin{corollary} \label{densScrd1}
Assume that $g \in \ScRd$ satisfies $\hatg(y) \neq 0$ for all $y \in \Rdst$.
Then for any finite set $M \subset \ScRd$ there is a sequence of  finite rank
operators $T_n$ with range in the linear span of
$ S(g_0):= \{T_{x_i}g_0, i \in F \}$, such that
$ T_n(f) \to f \,\, \mbox{in} \,\,\,\ScRd \quad \mbox{for} \,\,
\, n \to \infty$, for each $f \in M$.
In particular, this linear span is dense in $\ScRd$.
\end{corollary}
\begin{proof}
We only have to observe that the family of norms for the spaces
$\Msp^\infty_{v_s}(\Rdst)$, for $s \geq 0$ define a topology which is equivalent
to the usual topology on $\ScRd$ (see also \cite{grzi04}).
 Also, $\ScRd$ is not only
the intersection of all these spaces, but one can also
replace them by the closure of $\ScRd$ in each of these space
(for each fixed $s \geq 0$). Since such spaces are typical
examples for the setting of our main result we can guarantee
convergence of a suitable sequence of finite rank operators
for each of these norms, even uniformly with respect to compact
subsets of such a space.
\end{proof}

The extension of the statement to relatively compact subsets of $\ScRd$
is just a matter of technical arguments which are beyond the focus of
our paper. Also, it is true that one can find easier, non-constructive
arguments for the last statement, making use of standard Fourier
transform methods for the space of tempered distributions.


\section{Acknowledgement}
{This work was initiated during the visit of the
second author to the NuHAG work-group at the University of Vienna,
supported by an Ernst Mach Grant-Worldwide Fellowship (ICM-2019-13302)
from the OeAD-GmbH, Austria. The second author is very grateful to
professor Hans G. Feichtinger for his guidance, for the kind hospitality and
arranging excellent research facilities at the University of Vienna.
She was supported by NBHM-DAE (0204/19/2019R\&D-II/10472), India and the Austrian Science Fund (FWF) project  TAI6.}

\tred{The authors are grateful to the reviewer of an earlier version
of the manuscript for a very critical
reading of this manuscript which finally led to significant
improvement of the main result and a widened scope. One of the
significant changes undertaken now is the modification from
the approximation of individual functions to compact
sets $M$ in $\BspN$, using finite rank operators. 

\vspace{3mm}

\bibliographystyle{abbrv}  

\begin{thebibliography}{45}

\bibitem{baparo72}
G.~F. {B}achelis, W.~A. {P}arker, and K.~A. {R}oss.
\newblock {L}ocal units in ${L}^1({G})$.
\newblock {\em Proc. Amer. Math. Soc.}, 31:312--313, 1972.

\bibitem{be38}
 A.~{B}eurling:  \newblock {S}ur les int$\acute{e}$grales de Fourier absolutment
  convergentes et leur application $\grave{a}$ une transformation fonctionelle,
\newblock{ \em IX. Congr. Math. Scand., pp.345-366, Helsingfors, 1938.}

\bibitem{beok20}
A.~{B}enyi and K.~A. {O}koudjou.
\newblock {\em {M}odulation {S}paces}.
\newblock {S}pringer, {B}irkh{\"a}user, {N}ew {Y}ork, 2020.

\bibitem{brfe83}
W.~{B}raun and H.~G. {F}eichtinger.
\newblock {B}anach spaces of distributions having two module structures.
\newblock {\em J. Funct. Anal.}, 51:174--212, 1983.

\bibitem{bu72}
J.~T. {B}urnham.
\newblock {C}losed ideals in subalgebras of {B}anach algebras {I}.
\newblock {\em Proc. Amer. Math. Soc.}, 32:551--555, 1972.

\bibitem{ci69}
J.~{C}igler.
\newblock {N}ormed ideals in ${L}^1({G})$.
\newblock {\em Nederl. Akad. Wetensch. Proc. Ser. A}, {S}er. {A}(74):273--282,
  1969.

\bibitem{coro20}
E.~{C}ordero and L.~{R}odino.
\newblock {\em {T}ime-frequency {A}nalysis of {O}perators}.
\newblock {D}e {G}ruyter {S}tudies in {M}athematics, {B}erlin, 2020.

\bibitem{dipiprvi19}
P.~{D}imovski, S.~{P}ilipovic, B.~{P}rangoski, and J.~{V}indas.
\newblock {T}ranslation--modulation invariant {B}anach spaces of
  ultradistributions.
\newblock {\em J. Fourier Anal. Appl.}, 25(3):819--841, 2019.

\bibitem{dipivi15-1}
P.~{D}imovski, S.~{P}ilipovic, and J.~{V}indas.
\newblock {N}ew distribution spaces associated to translation-invariant
  {B}anach spaces.
\newblock {\em Monatsh. Math.}, 177(4):495--515, 2015.

\bibitem{do56}
Y.~{D}omar.
\newblock {H}armonic analysis based on certain commutative {B}anach algebras.
\newblock {\em Acta Math.}, 96:1--66, 1956.

\bibitem{du74}
D.~H. {D}unford.
\newblock {S}egal algebras and left normed ideals.
\newblock {\em J. Lond. Math. Soc. (2)}, 8:514--516, 1974.

\bibitem{fe77}
H.~G. {F}eichtinger.
\newblock {R}esults on {B}anach ideals and spaces of multipliers.
\newblock {\em Math. Scand.}, 41(2):315--324, 1977.

\bibitem{fe79}
H.~G. {F}eichtinger.
\newblock {G}ewichtsfunktionen auf lokalkompakten {G}ruppen.
\newblock {\em {S}itzber. d. {\"o}sterr. {A}kad. {W}iss.}, 188:451--471, 1979.

\bibitem{fe81}
H.~G. {F}eichtinger.
\newblock {A} characterization of minimal homogeneous {B}anach spaces.
\newblock {\em Proc. Amer. Math. Soc.}, 81(1):55--61, 1981.

\bibitem{fe83}
H.~G. {F}eichtinger.
\newblock {B}anach convolution algebras of {W}iener type.
\newblock In {\em  Proc. Conf. on Functions, Series, Operators, Budapest 1980},  p.509--524, Vol.35, North-Holland, Colloq. Math. Soc. Janos Bolyai, [Amsterdam] 1983.


\bibitem{fe84}
H.~G. {F}eichtinger.
\newblock {C}ompactness in translation invariant {B}anach spaces of
  distributions and compact multipliers.
\newblock {\em J. Math. Anal. Appl.}, 102:289--327, 1984.

\bibitem{fe87-1}
H.~G. {F}eichtinger.
\newblock {M}inimal {B}anach spaces and atomic representations.
\newblock {\em Publ. Math. Debrecen}, 34(3-4):231--240, 1987.

\bibitem{fe88}
H.~G. {F}eichtinger.
\newblock {A}n elementary approach to {W}iener's third {T}auberian theorem for
  the {E}uclidean $n$-space.
\newblock In {\em {S}ymposia {M}ath.}, Vol. {X}{X}{I}{X} of {\em {A}nalisa
  {A}rmonica}, pages 267--301, {C}ortona, 1988.

H.~G. {F}eichtinger and K.~{G}r{\"o}chenig.
\newblock {B}anach spaces related to integrable group representations and their
atomic decompositions, {I}.
\newblock {\em J. Funct. Anal.}, 86(2):307--340, 1989

 \bibitem{fe91-1}
H.~G. {F}eichtinger.
\newblock {D}iscretization of convolution and reconstruction of band-limited
functions from irregular sampling.
\newblock pages 333--345. {A}cademic {P}ress, {B}oston, {M}{A}, 1991.

\bibitem{fe03-1}
H.~G. {F}eichtinger.
\newblock {M}odulation spaces on locally compact {A}belian groups.
\newblock In R.~{R}adha, M.~{K}rishna, and S.~{T}hangavelu, editors, {\em
  {P}roc. {I}nternat. {C}onf. on {W}avelets and {A}pplications}, pages 1--56,
  {C}hennai, {J}anuary 2002, 2003. {N}ew {D}elhi {A}llied {P}ublishers.

  \bibitem{fe15}
H.~G. {F}eichtinger.
\newblock {\em {C}hoosing {F}unction {S}paces in {H}armonic {A}nalysis},
volume~4 of {\em {T}he {F}ebruary {F}ourier {T}alks at the {N}orbert {W}iener
{C}enter, {A}ppl. {N}umer. {H}armon. {A}nal.}, pages 65--101.
\newblock {B}irkh{\"a}user/{S}pringer, {C}ham, 2015.

\bibitem{fe16}
H.~G. {F}eichtinger.
\newblock {A} novel mathematical approach to the theory of translation
invariant linear systems.
\newblock In {P}eter {J}.~{B}entley and I.~{P}esenson, editors, {\em {N}ovel
{M}ethods in {H}armonic {A}nalysis with {A}pplications to {N}umerical
{A}nalysis and {D}ata {P}rocessing}, pages 483--516. {B}irkh{\"a}user,
{C}ham, 2017.

\bibitem{fefe19}
H.~G. {F}eichtinger and J.~{F}euto.
\newblock {P}redual of {F}ofana{\'{ }s} spaces.
\newblock {\em {M}athematics ({M}{D}{P}{I})}, 7(6):528, 2019.

\bibitem{fe22}
H.~G. {F}eichtinger.
\newblock {H}omogeneous {B}anach spaces are {B}anach convolution
modules over ${M(G)}$. MDPI, Mathematics. Vol. 10 No.3, 2022, 1-22.

\bibitem{fegr85}
H.~G. {F}eichtinger and P.~{G}r{\"o}bner.
\newblock {B}anach spaces of distributions defined by decomposition methods.
{I}.
\newblock {\em Math. Nachr.}, 123:97--120, 1985.

\bibitem{fegr89}
H.~G. {F}eichtinger and K.~{G}r{\"o}chenig.
\newblock {B}anach spaces related to integrable group representations and their
atomic decompositions, {I}.
\newblock {\em J. Funct. Anal.}, 86(2):307--340, 1989 

\bibitem{fegu20}
H.~G. {F}eichtinger and A.~{G}umber.
\newblock {C}ompleteness of shifted dilates in invariant {B}anach spaces of
  tempered distributions.
\newblock {\em Proc. Amer. Math. Soc.}, 149(12):5195--5210., 08 2021.

\bibitem{fezi02}
H.~G. {F}eichtinger and G.~{Z}immermann.
\newblock {A}n exotic minimal {B}anach space of functions.
\newblock {\em Math. Nachr.}, 239-240:42--61, 2002.

\bibitem{gr01}
K.~{G}r{\"o}chenig.
\newblock {\em {F}oundations of {T}ime-{F}requency {A}nalysis}.
\newblock {A}ppl. {N}umer. {H}armon. {A}nal. {B}irkh{\"a}user, {B}oston,
{M}{A}, 2001.



\bibitem{gr07}
K.~{G}r{\"o}chenig.
\newblock {W}eight functions in time-frequency analysis.
\newblock In L.~{R}odino and et~al., editors, {\em {P}seudodifferential
  {O}perators: {P}artial {D}ifferential {E}quations and {T}ime-{F}requency
  {A}nalysis}, volume~52 of {\em {F}ields {I}nst. {C}ommun.}, pages 343--366.
  {A}mer. {M}ath. {S}oc., {P}rovidence, {R}{I}, 2007.

\bibitem{grheok02}
K.~{G}r{\"o}chenig, C.~{H}eil, and K.~{O}koudjou.
\newblock {G}abor analysis in weighted amalgam spaces.
\newblock {\em Sampl. Theory Signal Image Process.}, 1(3):225--259, 2002.

\bibitem{grzi04}
K.~{G}r{\"o}chenig and G.~{Z}immermann.
\newblock {S}paces of test functions via the {S}{T}{F}{T}.
\newblock {\em J. Funct. Spaces Appl.}, 2(1):25--53, 2004.

\bibitem{he03}
C.~{H}eil.
\newblock {A}n introduction to weighted {W}iener amalgams.
\newblock In M.~{K}rishna, R.~{R}adha, and S.~{T}hangavelu, editors, {\em
  {W}avelets and their {A}pplications ({C}hennai, {J}anuary 2002)}, pages
  183--216. {A}llied {P}ublishers, {N}ew {D}elhi, 2003.

\bibitem{hero70}
E.~{H}ewitt and K.~A. {R}oss.
\newblock {\em {A}bstract {H}armonic {A}nalysis. {V}ol. {I}{I}: {S}tructure and
  {A}nalysis for {C}ompact {G}roups. {A}nalysis on {L}ocally {C}ompact
  {A}belian {G}roups}.
\newblock {S}pringer, {B}erlin, {H}eidelberg, {N}ew {Y}ork, 1970.

\bibitem{ka19-1}
V.~{K}atsnelson.
\newblock {O}n the completeness of {G}aussians in a {H}ilbert functional space.
\newblock {\em Complex Anal. Oper. Theory}, 13(3):637--658, 2019.

\bibitem{ko04}
J.~{K}orevaar.
\newblock {\em {T}auberian {T}heory. {A} {C}entury of {D}evelopments.}
\newblock {S}pringer, {B}erlin, 2004.
{ {\bibitem{pe76} J.~{P}eetre. \newblock {\em
{N}ew {T}houghts on {B}esov {S}paces}.
\newblock {D}uke {U}niversity {M}athematics {S}eries, {N}o. 1,  {D}uke {U}niversity, 1976. vi+305 pp.}}

\bibitem{re68}
H.~{R}eiter.
\newblock {\em {C}lassical {H}armonic {A}nalysis and {L}ocally {C}ompact
  {G}roups}.
\newblock {C}larendon {P}ress, {O}xford, 1968.

\bibitem{rest00}
H.~{R}eiter and J.~D. {S}tegeman.
\newblock {\em {C}lassical {H}armonic {A}nalysis and {L}ocally {C}ompact
  {G}roups. 2nd ed.}
\newblock {C}larendon {P}ress, {O}xford, 2000.

\bibitem{ri69-1}
M.~A. {R}ieffel.
\newblock {M}ultipliers and tensor products of ${L}^{p}$-spaces of locally
compact groups.
\newblock {\em Studia Math.}, 33:71--82, 1969.


\bibitem{tr83}
H.~{T}riebel.
\newblock {\em {T}heory of {F}unction {S}paces.} Vol.~78 of {\em
{M}onographs in {M}athematics}.
\newblock {B}irkh{\"a}user, {B}asel, 1983.



\bibitem{wa74}
W.~{W}alter.
\newblock {\em {E}inf{\"u}hrung in die {T}heorie der {D}istributionen}.
\newblock {B}.{I}. {T}aschenbuch, 1974.

\bibitem{wi32}
N.~{W}iener.
\newblock {T}auberian theorems.
\newblock {\em Ann. of Math. (2)}, 33(1):1--100, 1932.

\bibitem{wi33}
N.~{W}iener.
\newblock {\em {T}he {F}ourier {I}ntegral and Certain of its {A}pplications}.
\newblock {C}ambridge {U}niversity {P}ress, {C}ambridge, 1933.

\end{thebibliography}


\end{document}